\documentclass[12pt,a4paper,reqno]{amsart}
\usepackage{amsmath, amsfonts,mathrsfs,amssymb,amscd,graphicx,setspace}
\usepackage[mathscr]{eucal}
\usepackage{mathtools}
\usepackage{hyperref,url}
\usepackage{enumitem}
\usepackage{color,soul} 

\newtheorem{theorem}{Theorem}[section]

\newtheorem{remark}[theorem]{Remark}

\newtheorem{lemma}[theorem]{Lemma}
\newtheorem{definition}[theorem]{Definition}

\newcommand{\thmref}[1]{Theorem~\ref{#1}}

\allowdisplaybreaks

\begin{document}

\title
{On the coefficients of symmetric power $L$-functions}
\author{Jaban Meher, Karam Deo Shankhadhar and G. K. Viswanadham}

\address[Jaban Meher]{School of Mathematical Sciences, National Institute of Science Education and Research, HBNI, Bhubaneswar, P.O. Jatni, Khurda 752 050, Odisha, India.}
\email{jaban@niser.ac.in}

\address[Karam Deo Shankhadhar]{Department of Mathematics, Indian Institute of Science Education and Research Bhopal, Bhopal Bypass Road, Bhauri, Bhopal 462 066, Madhya Pradesh, India.}
\email{karamdeo@iiserb.ac.in}

\address[G. K. Viswanadham]{Department of Mathematics, Indian Institute of Technology Bombay, Mumbai 400 076, Maharashtra, India.}
\email{vissu35@gmail.com}



\begin{abstract}
We study the signs of the Fourier coefficients of a newform. Let $f$ be a normalized newform of weight $k$ for $\Gamma_0(N)$. Let $a_f(n)$ be the $n$th Fourier coefficient of $f$. For any fixed positive integer $m$, we study the distribution of the signs of $\{a_f(p^m)\}_p$, where $p$ runs over all prime numbers. We also find out the abscissas of absolute convergence of two Dirichlet series with coefficients involving the Fourier coefficients of cusp forms and the coefficients of symmetric power 
$L$-functions. 
\end{abstract}

\maketitle

\section{Introduction}

Let $S_k(N)$ denote the space of cusp forms of even integral weight $k$ for the group 
$\Gamma_0(N)$. 
Let $f\in S_k(N)$ be a normalized newform with Fourier expansion  
\begin{equation}\label{fourier}
f(z)=\sum_{n=1}^{\infty}a_f(n)e^{2\pi i nz}\; =\; \sum_{n=1}^{\infty}n^{\frac{k-1}{2}}\lambda_f(n)e^{2\pi inz},
\end{equation}
where $z$ is in the complex upper-half plane $\mathcal{H}$.
It is well known that the Fourier coefficients $\lambda_f(n)$ are real numbers. 
The signs of the Fourier coefficients $\lambda_f(n)$ have been studied by several authors due to their various number theoretic applications. A standard result of the classical theorem of Landau on Dirichlet series with non-negative coefficients implies that the sequence $\{\lambda_f(n)\}_{n\geq 1}$ changes sign infinitely often. 
It is natural to consider the signs of $\lambda_f(n)$ where $n$ varies over a sparse subset $\mathcal{S}$ of the set of natural numbers. The case when $\mathcal{S}$ is the set of prime numbers was studied by M. Ram Murty \cite{Ram} in detail. When $\mathcal{S}$ is the set of squares, cubes or fourth powers of the natural numbers, the sign changes were studied by the authors in \cite{JKV}. 
Using \cite[Corollary 3]{MR} about sign change of multiplicative functions established by 
K. Matom{\"a}ki and M. Radziwi{\l}{\l} one deduces that for any fixed positive integer $m$, the sequence 
$\{\lambda_f(n^m)\}_{n\geq 1}$ has infinitely many sign changes, and the number of sign changes up to $x$ is $\gg x$. 
Let $m$ be a fixed positive integer. In this article first we discuss the distribution of the signs of 
$\{\lambda_f(p^m)\}$, where $p$ runs over all prime numbers. One has 
$\lambda_f(p^m)=\lambda_{{\rm sym}^m f}(p)$, where $\lambda_{{\rm sym}^m f}(p)$ denotes the 
$p$-th coefficient of the  $m$-th symmetric power $L$-function $L(s, sym^m f)$ attached to $f$. 

Let us denote the set of all primes by $\mathbb{P}$. For a subset $A \subseteq \mathbb{P}$, the density of $A$
in $\mathbb{P}$, denoted by $d(A)$, is defined as
\begin{equation}
d(A)\;=\; \lim_{x\to \infty}\frac{\#\{p\leq x\mid p\in A\}}{\#\{p\leq x \mid p \in \mathbb{P}\}},
\end{equation} 
provided the limit exists.
For a fixed integer $m\geq 1$, let 
$$
P_m :=\{p \in \mathbb{P} \mid p \nmid N, \lambda_f(p^m)>0\},
$$
and 
$$
P^{'}_{m} :=\{p \in \mathbb{P} \mid p \nmid N, \lambda_f(p^m)<0\}.
$$
In our first two theorems which are stated below, we calculate the densities of the two sets $P_m$ and $P^{'}_{m}$. From the two theorems, we see that the densities depend on whether a form is with complex multiplication or without complex multiplication as well as the parity of $m$.
\begin{theorem}\label{thm:noncm}
Let $m\geq 1$ be an integer.
Let $f\in S_k(N)$ be a normalized newform without complex multiplication. 
Then the sequence $\{\lambda_f(p^m)\}_{p \in \mathbb{P}}$ changes signs infinitely often.
Moreover, we have 
\begin{equation}
    d(P_m)=
    \begin{cases*}
      \frac{1}{2} & \rm{if} $m$ is odd, \\
      \frac{m+2}{2(m+1)} -\frac{1}{2\pi}\tan{(\frac{\pi}{m+1})} & \rm{if} $m$ is even,
    \end{cases*}
  \end{equation}
  and
  \begin{equation}
    d(P^{'}_{m})=
    \begin{cases*}
      \frac{1}{2} & \rm if $m$ is odd,\\
      \frac{m}{2(m+1)} +\frac{1}{2\pi}\tan{(\frac{\pi}{m+1})} & \rm if $m$ is even.
    \end{cases*}
  \end{equation}
\end{theorem}
\begin{theorem}\label{thm:cm}
Let $m\geq 1$ be an integer.
Let $f\in S_k(N)$ be a normalized newform with complex multiplication. Then 
the sequence $\{\lambda_f(p^m)\}_{p \in \mathbb{P}}$ changes signs infinitely often.
Moreover, we have  
\begin{equation}
    d(P_m)=
    \begin{cases*}
      \frac{1}{4} & \rm if $m$ is odd, \\
      \frac{m+2}{4(m+1)}+\frac{1}{2}    & \rm if $m\equiv0\pmod 4$, \\
      \frac{m+2}{4(m+1)} & \rm if $m\equiv 2\pmod 4$,
    \end{cases*}
  \end{equation}
  and 
  \begin{equation}
    d(P^{'}_m)=
    \begin{cases*}
      \frac{1}{4} & \rm if $m$ is odd, \\
      \frac{m}{4(m+1)} & \rm if $m\equiv 0\pmod 4$,\\
      \frac{m}{4(m+1)} +\frac{1}{2}   & \rm if $m\equiv 2\pmod 4$.
    \end{cases*}
  \end{equation}
\end{theorem}
Our next result is on the abscissa of absolute convergences of two Dirichlet series attached to $f$.
Assume that $f$ is a normalized Hecke eigenform of even integral weight $k$ for the full modular group  $SL_2(\mathbb{Z})$ with Fourier expansion given in \eqref{fourier}. For a fixed integer $m\ge 1$, let
\begin{equation}\label{l_1(s)}
L_m(s,f) := \sum_{n \geq 1}\lambda_f(n^m) n^{-s}\; 
 \end{equation}
and
\begin{equation}\label{l_2(s)}
L(s, sym^m f) := \sum_{n \geq 1} \lambda_{{\rm sym}^m f}(n) n^{-s}.
\end{equation}
We prove the following theorem.
\begin{theorem}\label{thm:abscissa_absolute}
For any positive integer $m\ge 1$, each of the Dirichlet series $L_m(s, f)$ and $L(s, sym^m f)$ defined in \eqref{l_1(s)} and \eqref{l_2(s)}, has abscissa of absolute convergence exactly equal to $1$. 
\end{theorem}
\begin{remark}
The abscissa of absolute convergence of the series $L(s, sym^m f)$ has been found out in \cite{KMP}. But the method of proof of the same result is different in this paper. However, using the method of \cite{KMP}, one will not be able to find the abscissa of absolute convergence of the series $L_m(s, f)$ for all integers $m\ge 1$.
\end{remark}
\begin{remark}
We feel that it may be possible to prove \thmref{thm:abscissa_absolute} for higher level case. But in the proof of \thmref{thm:abscissa_absolute}, we use the result of \cite{TW} which is proved only in level $1$ case.
\end{remark}
 
This article is organised as follows. In the next section we provide the formulas for certain Fourier coefficients. We recall certain results regarding the distribution of the Fourier coefficients and  
write down certain trigonometric formulas for the prime power Fourier coefficients by using the Hecke relation among the coefficients. Also, we recall certain theorems including Sato-Tate conjecture which will be used in establishing our results. In Section 3, we prove Theorem \ref{thm:noncm} by using the distribution results for forms without complex multiplication. In Section 4, we prove Theorem \ref{thm:cm} by using the distribution results for forms with complex multiplication. 
In Section 5, we prove Theorem \ref{thm:abscissa_absolute} by using certain results established in \cite{TW}.

\section{Preliminary results}

Let 
$$f=\sum_{n=1}^{\infty}n^{\frac{k-1}{2}}\lambda_f(n)e^{2\pi inz}\in S_k(N)$$ 
be a normalized newform. 
The Ramanujan-Petersson conjecture which was proved by Deligne, states that
$$|\lambda_f(p)|\le 2,$$ 
where $p$ is any prime number not dividing $N$.
Thus for any prime  $p, p \nmid N$, we have
$$
\lambda_f(p)= 2\cos{\theta_p}
$$
for some $\theta_p\in [0,\pi]$. The next result which is well-known, gives a trigonometric formula for 
$\lambda_f(p^m)$. For the sake of completeness we provide a proof of the result.

\begin{lemma}\label{lemma:lambdaformula}
Let $p$ be any prime number not dividing $N$ and $m\geq 1$ be any positive integer. We have
\begin{equation}\label{eq:lambdapm}
\lambda_f(p^m)\;=\; \frac{\sin{((m+1)\theta_p)}}{\sin{\theta_p}}\;,
\end{equation}
with the interpretation that the values of $\lambda_f(p^m)$ are $m+1$ and $(-1)^m (m+1)$ when the values of $\theta_p$ are $0$ and $\pi$ respectively.
\end{lemma}

\begin{proof}
To establish the above lemma we use induction on $m$. The lemma is true for $m=1$ since
$\sin{2\theta_p}=2 \sin{\theta_p}\cos{\theta_p}$.
Assume that the lemma is true for all positive integers $\le m$. 
Using the Hecke relation among the Fourier coefficients, we have
$$
\lambda_f(p^{m+1})=\lambda_f(p^m)\lambda_f(p)-\lambda_f(p^{m-1}).
$$
Applying the induction hypothesis, we have
$$
\lambda_f(p^{m+1})=\frac{\sin{((m+1)\theta_p)}}{\sin{\theta_p}}
2\cos{\theta_p}-\frac{\sin{m\theta_p}}{\sin{\theta_p}}.
$$
Using the trigonometric formulas 
$$\sin{((m+1)\theta_p)}=\sin{m\theta_p}\cos{\theta_p}+\cos{m\theta_p}\sin{\theta_p},$$
and
$$
\cos{2\theta_p}=2\cos^2{\theta_p}-1
$$ 
in the above expression, we obtain
$$
\lambda_f(p^{m+1})=\frac{\sin{m\theta_p}\cos{2\theta_p}+
\cos{m\theta_p}\sin{2\theta_p}}{\sin{\theta_p}}=\frac{\sin{((m+2)\theta_p)}}{\sin{\theta_p}}.
$$
This proves the lemma by induction.
\end{proof}

\begin{definition}\label{defn:satotatemeasure}{\rm (Sato-Tate measure)}
The Sato-Tate measure $\mu_{ST}$ is the probability measure on $[0,\pi]$ given by 
$\frac{2}{\pi}\sin^2\theta d\theta$.
\end{definition}
At this point we state certain crucial results which will be used to establish \thmref{thm:noncm} and \thmref{thm:cm}.
For any interval $I\subseteq [0,\pi]$, let 
$$\pi_I(x)=\# \{p \leq x \mid p\in\mathbb{P}, p\nmid N, \theta_p\in I\}.$$
 Taking $\zeta=1$ in case 3 of \cite[Theorem B]{BGHT}, we get the following equidistribution result for newforms without complex multiplication.
\begin{theorem}\label{thm:satotate}{\rm (Barnet-Lamb, Geraghty, Harris, Taylor)}
Let $f\in S_k(N)$ be a normalized newform without complex multiplication. The sequence 
$\{\theta_p\}_{p\nmid N}$ is equidistributed in $[0,\pi]$ with respect to the Sato-Tate measure $\mu_{ST}$. In particular, for any sub-interval $I\subseteq[0,\pi]$ we have
$$
\lim_{x\to 
\infty}\frac{\pi_I(x)}{\pi(x)}=\mu_{ST}(I)=\frac{2}{\pi}\int_I\sin^2{\theta} d\theta.
$$
\end{theorem}
The above theorem implies that if $A$ is a finite set, then the density of the set 
$\{p\in \mathbb{P} ~~~|~~~ \theta_p\in A\}$ is $0$.
For any interval $I$, let $\vert I \vert$ denote the length of the interval $I$. The following theorem is the analog of Theorem \ref{thm:satotate} for the newforms with complex multiplication. For a proof we refer to \cite[Theorem 3.1.1 (a)]{AIW}.
\begin{theorem}\label{thm:deuring}{\rm (Deuring equi-distribution)}
Let $f\in S_k(N)$ be a newform with complex multiplication.
Let $I\subseteq [0,\pi]$ be an interval such that  $\frac{\pi}{2}\notin I$. Then we have 
\[
\lim_{x\to \infty}\frac{\pi_I(x)}{\pi(x)•}\;=\; \frac{\vert I\vert}{2\pi},
\]
and the set $\{p \ {\rm prime} \mid \theta_p=\frac{\pi}{2}\}$ has density $\frac{1}{2}$ in the set of primes.
\end{theorem}
In this case, if $A$ is a finite set not containing $\pi/2$, then by \thmref{thm:deuring} the density of the set $\{p\in \mathbb{P} ~~~|~~~ \theta_p\in A\}$ is $0$.
Next, we state the following theorem established by Tang and Wu \cite{TW}, which will be used to prove 
Theorem \ref{thm:abscissa_absolute}.
\begin{theorem}\label{thm:TW}
We have 
\begin{equation}\label{eq:3}
\begin{split}
\sum_{n\leq x} \vert \lambda_{{\rm sym}^m f}(n) \vert & \sim C_m(f) x (\log 
x)^{-\delta_m},\\
\sum_{n\leq x} \vert \lambda_f(n^m) \vert & \sim D_m(f) x (\log 
x)^{-\delta_m},
\end{split}
\end{equation}
unconditionally for $x \rightarrow \infty$, where $C_m(f), D_m(f)$ are two positive 
constants depending on $f, m$ and
\begin{equation*}
\delta_m = 1-\frac{4 (m+1)}{\pi m(m+2)} \cot \left( \frac{\pi}{2(m+1)}\right).
\end{equation*}
\end{theorem}

\section{Proof of Theorem \ref{thm:noncm}}
By Lemma \ref{lemma:lambdaformula}, we have
$$\lambda_f(p^m)= \frac{\sin{((m+1)\theta_p)}}{\sin{\theta_p}}, \ \theta_p \in [0,\pi],$$
with the obvious interpretation in the limiting case $\theta_p=0, \pi$.
Since the set $\{p\in \mathbb{P} ~~~|~~~ \theta_p=0 ~~~{\rm or}~~~\pi\}$ has density $0$,
we may assume that $\theta_p$ is different from $0$ and $\pi$. 
For $\theta_p$ in the interval $(0,\pi)$, we know that $\sin{\theta_p} > 0$. Therefore the sign of 
$\lambda_f(p^m)$ is the same as the sign of $\sin{((m+1) \theta_p)}$.
\subsection{Even case}
First assume that $m$ is even. Then we have 
\begin{equation}\label{eq:1}
\begin{split}
\sin{((m+1) \theta_p)} >0 & \iff
(m+1)\theta_p\in \bigcup_{j=0}^{\frac{m}{2}}(2j\pi , (2j+1)\pi)\\
& \iff
\theta_p\in \bigcup_{j=0}^{\frac{m}{2}}\left(\frac{2j\pi}{m+1}, \frac{(2j+1)\pi}{m+1}\right).
\end{split}
\end{equation}
Similarly,
\begin{equation}\label{eq:2}
\sin{((m+1) \theta_p)} <0 \iff
\theta_p\in \bigcup_{j=1}^{\frac{m}{2}}\left(\frac{(2j-1)\pi}{m+1}, \frac{2j\pi}{m+1}\right).
\end{equation}
From Theorem \ref{thm:satotate} we know that the sequence $\{\theta_p\}_{p\in \mathbb{P}}$ with
$p\not|N$ is 
equi-distributed in $[0,\pi]$ with respect to the Sato-Tate measure $\mu_{ST}$.
Therefore there exist infinitely many primes $p$ such that 
$\theta_p\in \bigcup_{j=0}^{\frac{m}{2}}\left(\frac{2j\pi}{m+1}, \frac{(2j+1)\pi}{m+1}\right)$, and there exist infinitely many primes $p$ such that 
$\theta_p\in \bigcup_{j=1}^{\frac{m}{2}}\left(\frac{(2j-1)\pi}{m+1}, \frac{2j\pi}{m+1}\right)$.
Thus the sequence $\{\lambda_f(p^m)\}_{p\in \mathbb{P}}$ changes signs infinitely often. 
Next we calculate the densities of the sets $P_m$ and $P^{'}_m$.
Let
$$
A=\bigcup_{j=0}^{\frac{m}{2}}\left(\frac{2j\pi}{m+1}, \frac{(2j+1)\pi}{m+1}\right)
 \subseteq [0,\pi].
$$
From the definitions of the density $d$ and Sato-Tate measure $\mu_{ST}$, we have
$$
d(P_m)= \mu_{ST}(A)= \frac{2}{\pi}\int_{A}\sin^2t \ dt= 
\frac{2}{\pi}\sum_{j=0}^{\frac{m}{2}}\int_{\left(\frac{2j\pi}{m+1},\frac{(2j+1)\pi}{
m+1}\right)}\sin^2t \ dt.
$$
Using the fact $\int\sin^2t \ dt= \frac{t}{2}-\frac{\sin 2t}{4}$ in the above expression, we obtain
$$
d(P_m)=\frac{2}{\pi}\sum_{j=0}^{\frac{m}{2}}
\left(\frac{\pi}{2(m+1)}-\frac{\sin{\left(\frac{(2j+1)2\pi}{m+1}\right)}}{4}
+\frac{\sin{\left(\frac{4j\pi}{m+1}\right)}}{4}\right).
$$
Denoting $\frac{2\pi}{m+1•}$ by $x$ in the above expression, we get
\begin{equation}\label{j1}
d(P_m)=\frac{m+2}{2(m+1)}-\frac{1}{2\pi}\{\sin{x}-\sin{2x}+\sin{3x}-\cdots+
\sin{((m+1)x)}\}.
\end{equation}
Since $\sin{((m+1)x)}=0$, the trigonometric term $\sin{x}-\sin{2x}+\sin{3x}-\cdots+\sin{((m+1)x)}$ 
in \eqref{j1} can be rewritten as
$$
\{\sin{x}+\sin{2x}+\sin{3x}+\cdots+\sin{mx}\}-2\{\sin{2x}+\sin{4x}+\cdots+\sin{mx}\}.
$$
Now using the trigonometric  identity
$$
\sin{x}+\sin{2x}+\cdots+\sin{nx}\;=\; \frac{\sin{\frac{(n+1)x}{2}}\sin{\frac{nx}{2}}}
{\sin{\frac{x}{2}}}
$$
in \eqref{j1},
we obtain
\[
d(P_m)=\frac{(m+2)}{2(m+1)}-\frac{1}{2\pi•}\left(\frac{\sin 
\frac{(m+1)x}{2•}\sin\frac{mx}{2•}}{\sin \frac{x}{2}•}-2\frac{\sin 
((\frac{m}{2•}+1)x)\sin\frac{mx}{2•}}{\sin x}\right)\; .
\]
Since $\sin \left(\frac{m+1•}{2•}\right)x=0$, we deduce that
\[
d(P_m)=\frac{m+2}{2(m+1)•}+\frac{1}{\pi}\frac{\sin 
((\frac{m}{2•}+1)x)\sin\frac{mx}{2•}}{\sin x}\;.
\]
Putting $x= \frac{2\pi}{m+1}$ in the above equation and simplifying it, we get 
\[
d(P_m)\;=\; \frac{m+2}{2(m+1)•}-\frac{1}{2\pi•}\tan\left(\frac{\pi}{m+1•}\right)\; 
.
\]

Since the density of the set 
$$
\{p \in \mathbb{P} \mid \lambda_f(p^m)=0\}
=\{p \in \mathbb{P} \mid \theta_p =\frac{i\pi}{m+1}, 1\le i\le m\}
$$
is $0$ 
in the set of primes, the density $d(P^{'}_m)$ of $P^{'}_m$ is $1-d(P_m)$. Thus
\[
d(P^{'}_m)\;=\; 
\frac{m}{2(m+1)•}+\frac{1}{2\pi}\tan\left(\frac{\pi}{m+1•}\right)\; .
\]

\subsection{Odd case}
We assume that $m$ is an odd positive integer. Using the arguments similar to the even case, we have
$$
\lambda_f(p^m)>0 ~~~~\mbox{if~~~and~~~~only~~~~if}~~~~
\theta_p\in \bigcup_{j=1}^{\frac{m+1}{2}}\left(\frac{(2j-2)\pi}{m+1}, \frac{(2j-1)\pi}{m+1}\right),
$$
and 
$$\lambda_f(p^m)<0 ~~~~\mbox{if~~~and~~~~only~~~~if}~~~~
\theta_p\in \bigcup_{j=1}^{\frac{m+1}{2}}\left(\frac{(2j-1)\pi}{m+1}, \frac{2j\pi}{m+1}\right).
$$
Let 
$$
A= \bigcup_{j=1}^{\frac{m+1}{2}}\left(\frac{(2j-2)\pi}{m+1}, \frac{(2j-1)\pi}{m+1}\right)
$$
and
$$
B=\bigcup_{j=0}^{\frac{m-1}{2}}\left(\frac{(2j+1)\pi}{m+1}, \frac{(2j+2)\pi}{m+1}\right)
=\bigcup_{j=1}^{\frac{m+1}{2}}\left(\frac{(m-2j+2)\pi}{m+1}, \frac{(m-2j+3)\pi}{m+1}\right).
$$
To prove that the sets $P_m$ and $P_m^{'}$ have same density $\frac{1}{2}$, 
we prove that the Sato-Tate measures of the two sets $A$ and $B$ are equal. Since
each of the sets $A$ and $B$ are union of $\frac{m+1}{2}$ disjoint intervals given in 
the above form, to prove that $\mu_{ST}(A)=\mu_{ST}(B)$, it is sufficient to prove that 
$\mu_{ST}(I_j)=\mu_{ST}(I_j^{'})$ for each $j$ with $1\le j\le \frac{m+1}{2}$, where
$$
I_j=\left(\frac{(2j-2)\pi}{m+1}, \frac{(2j-1)\pi}{m+1}\right)~~~~~~~~~~~~~
\mbox{and}~~~~~~~~~~~~~I_j^{'}=\left(\frac{(m-2j+2)\pi}{m+1}, \frac{(m-2j+3)\pi}{m+1}\right).
$$

Using the integral evaluation $\int \sin^2 t \ dt= \frac{t}{2}-\frac{\sin 2t}{4}$, we get 
\begin{equation}\label{j2}
\int_{I_j}\sin^2x \ dx=\frac{\pi}{2(m+1)}+\frac{1}{4}
\left(\sin\left(\frac{2(2j-2)\pi}{m+1}\right)-\sin\left(\frac{2(2j-1)\pi}{m+1}\right)\right).
\end{equation}
On the other hand we have
\begin{equation}\label{j3}
\int_{I_j^{'}}\sin^2x \ dx=\frac{\pi}{2(m+1)}+\frac{1}{4}
\left(\sin\left(\frac{2(m-2j+2)\pi}{m+1}\right)-\sin\left(\frac{2(m-2j+3)\pi}{m+1}\right)\right).
\end{equation}
We know that
$$
\sin{\left(\frac{2(m-2j+2)\pi}{m+1}\right)}=\sin{\left(2\pi-\frac{2(2j-1)\pi}{m+1}\right)}
=-\sin{\left(\frac{2(2j-1)\pi}{m+1}\right)}.
$$
Similarly, 
$$\sin\left(\frac{2(m-2j+3)\pi}{m+1}\right)=-\sin\left(\frac{2(2j-2)\pi}{m+1}\right).$$
Thus the expressions in the right hand side of \eqref{j2} and \eqref{j3} are same.
Therefore from the above discussion, we deduce that
$$
\int_{I_j}\sin^2x \ dx=\int_{I_j^{'}}\sin^2x \ dx.
$$
This proves the theorem.
\section{Proof of Theorem \ref{thm:cm}}
Let $f$ be a normalized newform with complex multiplication.
\subsection{Odd case}
Let $m\ge1$ be an odd positive integer. If $\theta_p=\frac{\pi}{2}$, then
$$
\lambda_f(p^m)=\frac{\sin{\frac{(m+1)\pi}{2}}}{\sin{\frac{\pi}{2}}}=0.
$$
Therefore from Theorem \ref{thm:deuring}, we have
$$d(\{p \ {\rm prime} \mid \lambda_f(p^m)=0\})= d(\{p \ {\rm prime} \mid \theta_p=\pi/2\}) = \frac{1}{2}.$$ 
Following the arguments similar to the odd case in the previous section, we have 
$$
\lambda_f(p^m) > 0 \iff \theta_p \in 
\bigcup_{j=1}^{\frac{m+1}{2}}\left(\frac{(2j-2)\pi}{m+1}, \frac{(2j-1)\pi}{m+1}\right)
\bigg \backslash \{\pi /2\},
$$
and
$$
 \lambda_f(p^m)<0 \iff 
 \theta_p \in \bigcup_{j=1}^{\frac{m+1}{2}}\left(\frac{(2j-1)\pi}{m+1}, \frac{2j\pi}{m+1}\right)
\bigg \backslash \{\pi /2\}.
$$
Using Theorem \ref{thm:deuring}, we have
$$
d(P_m)=\frac{1}{2\pi}\sum_{j=1}^{\frac{m+1}{2}}\frac{\pi}{m+1}=\frac{1}{4}.
$$
Similarly, $d(P_m^{'})=\frac{1}{4}$.
\subsection{Even case}
Let $m\ge 1$ be an even integer. Then for primes $p$ such that $\theta_p=\frac{\pi}{2}$, we have 
\begin{equation*}
\lambda_f(p^m)= \sin{\frac{(m+1)\pi}{2}}=
    \begin{cases*}
      1 & if $m\equiv 0\pmod 4$, \\
      -1       & if $m\equiv 2\pmod 4$.
    \end{cases*}
  \end{equation*}
  From Theorem \ref{thm:deuring}, we have 
  \begin{equation}\label{eq:03}
  d(\{p \ {\rm prime} \mid \theta_p=\pi/2\}) = \frac{1}{2}.
  \end{equation}
  Suppose that $m\equiv 0 \pmod{4}$. 
  Following the arguments similar to even case in the previous section, we have 
$$
\lambda_f(p^m)>0  \iff 
 \theta_p \in \bigcup_{j=0}^{\frac{m}{2}}\left(\frac{2j\pi}{m+1}, \frac{(2j+1)\pi}{m+1}\right) 
 \bigcup \{\pi/2\},
 $$
and
$$
\lambda_f(p^m)<0 \iff 
 \theta_p \in \bigcup_{j=1}^{\frac{m}{2}}\left(\frac{(2j-1)\pi}{m+1}, \frac{2j\pi}{m+1}\right) 
 \bigg\backslash \{\pi/2\}.
$$
Now using Theorem \ref{thm:deuring} and \eqref{eq:03}, we have
$$
d(P_m)=\frac{m+2}{4(m+1)}+\frac{1}{2}
$$
and 
$$
 d(P^{'}_m)=\frac{m}{4(m+1)}.
$$
If $m\equiv 2 \pmod{4}$, then similarly we can prove that 
$$
d(P_m)=\frac{m+2}{4(m+1)}
$$
and 
$$
 d(P^{'}_m)=\frac{m}{4(m+1)} +\frac{1}{2}.
$$
This proves the theorem.

\section{Proof of Theorem \ref{thm:abscissa_absolute}}\label{abscissa}
First we prove the following result. If 
$$L(s) := \sum_{n=1}^{\infty} \frac{a(n)} {n^s}$$
is a Dirichlet series such that $\sum_{n \geq 1} | a(n) |$ diverges then 
the abscissa of absolute convergence 
$\sigma_{a}$ of 
$L(s)$ is given by
\begin{equation}\label{eq:4}
\sigma_a= \inf \left\{\alpha \in \mathbb{R} \mid \sum_{n \leq N} \vert a(n) 
\vert = O_\alpha(N^\alpha) \right\}.
\end{equation} 
Let
$$
\gamma= \inf \left\{\alpha \in \mathbb{R} \mid \sum_{n \leq N} \vert a(n) 
\vert = O_\alpha(N^\alpha) \right\},
$$
and 
$$
A(N)=\sum_{n=1}^{N}|a(n)|. 
$$
By partial summation we have
\begin{equation}\label{well}
\sum_{n=1}^{N}\frac{|a(n)|}{n^\beta}=A(N)\frac{1}{N^\beta}+\beta\int_{1}^N A(u)\frac{1}{u^{\beta+1}} du
\end{equation}
for any $\beta\in \mathbb{R}$.
If $\beta>\gamma$, then
$$
A(N)\frac{1}{N^\beta} \to 0 ~~~~{\rm as}~~~~~ N\to \infty
$$
and
$$
\int_{1}^\infty A(u)\frac{1}{u^{\beta+1}} du< \infty.
$$
Thus by taking $N\to \infty$ in both the sides of \eqref{well}, we deduce that 
if $\beta>\gamma$, the series
$$
\sum_{n=1}^{\infty}\frac{|a(n)|}{n^\beta}
$$
is convergent. If $\beta<\gamma$, then $N^\beta=o(A(N))$. Therefore
$$
A(N)\frac{1}{N^\beta} \to \infty ~~~~{\rm as}~~~~~ N\to \infty
$$
and
$$
\int_{1}^\infty A(u)\frac{1}{u^{\beta+1}} du
$$
is infinite. Thus by taking $N\to \infty$ in both the sides of \eqref{well}, we deduce that 
if $\beta<\gamma$, the series
$$
\sum_{n=1}^{\infty}\frac{|a(n)|}{n^\beta}
$$
is divergent. This proves that $\sigma_a=\gamma$.

Now we establish the abscissa of absolute convergence of the series $L_m(s, f)$. The establishment of the abscissa of absolute convergence for the series $L(s, sym^m f)$ can be treated exactly along the same lines as $L_m(s, f)$. First, observe that the series $\sum_{n \geq 1} \vert \lambda_f(n^m) \vert$ is divergent by the asymptotic formula given in \eqref{eq:3}.  Therefore the abscissa of absolute convergence $\sigma_{a,m}$ of $L_m(s,f)$ is equal to
$$
 \inf \left\{\alpha \in \mathbb{R} \mid \sum_{n \leq N} \vert \lambda_f(n^m) \vert 
= O_\alpha(N^\alpha) \right\}.
$$
For any $\epsilon>0$, we have  $ \lambda_f(n) = O_\epsilon(n^\epsilon)$ by Deligne bound. 
Using this bound, we have $\sigma_{a,m} \leq 1$. 
If $\sigma_{a,m} < 1$ then there exists  $\delta>0$ such that
\begin{equation}\label{eq:5}
\sum_{n \leq N} \vert \lambda_f(n^m) \vert = O_\delta (N^{1-\delta}).
\end{equation}
On the other hand, using the asymptotic formula \eqref{eq:3} we have
$$
\frac{N}{(\log N)^{\delta_m}} \ll \sum_{n \leq N} \vert \lambda_f(n^m) \vert \ 
\ \ \ {\rm as} \ N \rightarrow \infty.
$$
This is a contradiction to \eqref{eq:5}. Therefore
$\sigma_{a,m} =1$. 
In the case of $L(s,sym^m f)$, we need to use the bound 
$\lambda_{{\rm sym}^m f}(n) = O_\epsilon(n^\epsilon)$, which holds for each  $\epsilon >0$, and the 
asymptotic formula for the sum
$\sum_{n \leq x} \vert \lambda_{{\rm sym}^m f}(n) \vert$ presented in Theorem 
\ref{thm:TW} to get the required result. The bound $\lambda_{{\rm sym}^m f} (n) = O_\epsilon(n^{\epsilon})$ follows from the observation that 
$\vert \lambda_{{\rm sym}^m f} (n) \vert \leq d_{m+1}(n)$ \cite[Equation (1.8)]{TW},
where $d_{m+1}(n)$ denotes the number of ways one can write $n$ as a product of $m + 1$ natural numbers.\\

\noindent{\bf Acknowledgements.}  
We would like to thank Prof. M. Ram Murty for his valuable suggestions which were helpful in establishing the density results. We also thank Prof. J.-P. Serre for his comments on the paper and correcting one of the definitions. The research of the second author was partially supported by a DST-SERB grant ECR/2016/001359.

\end{document}